\documentclass[12pt]{amsart}
\usepackage{graphicx}
\usepackage{a4wide}
\usepackage{amssymb}
\usepackage{amsthm}
\usepackage{amsmath}
\usepackage{amscd} \usepackage{verbatim}
\usepackage[all]{xy}
\textheight8.25in \textwidth6.5in \numberwithin{equation}{section}
\theoremstyle{plain}
\newtheorem{theorem}{Theorem}[section]

\newtheorem{proposition}[theorem]{Proposition}

\theoremstyle{definition}

\newtheorem*{example}{Example}
\newtheorem*{problem}{Problem}

\theoremstyle{remark}
\newtheorem*{remark}{Remark}

\newcommand{\SL}{\text {\rm SL}}

\newcommand{\Z}{\mathbb{Z}}
\newcommand{\N}{\mathbb{N}}
\newcommand{\C}{\mathbb{C}}

\def\({\left(}
\def\){\right)}

\newcommand{\ord}{\text {\rm ord}}

\begin{document}

\title[A note on non-ordinary primes] {A note on non-ordinary primes}

\author{Seokho Jin, Wenjun Ma, and Ken Ono}

\address{School of Mathematics, Korea Institute for Advanced Study, Hoegiro 85, Dongdaemun-gu, Seoul 130-722, Korea}

\email{seokhojin@kias.re.kr}

\address{School of Mathematics, Shandong University, Jinan, Shandong, China. 250100}
\email{wenjunma.sdu@hotmail.com}

\address{Department of Mathematics,  Emory University,
Atlanta, GA. 30322}
\email{ono@mathcs.emory.edu}

\thanks{The first author thanks KIAS for its generous support. The second author thanks the China Scholarship Council for its generous support. The third author thanks the NSF and the Asa Griggs Candler Fund for their generous support.}
%%%%%%%%%%%%%%%%%%%%%%%%%%%%%

\begin{abstract} Suppose that $O_L$ is the ring of integers of a number field $L$, and suppose that
$$
f(z)=\sum_{n=1}^{\infty}a_f(n)q^n\in S_k\cap O_L[[q]]
$$
(note: $q:=e^{2\pi i z}$) is a normalized Hecke eigenform for $\SL_2(\Z)$.
We say that $f$ is non-ordinary at a prime $p$ if there is a prime ideal $\frak{p}\subset O_L$ above $p$ for which
$$
a_f(p)\equiv 0\pmod{\frak{p}}.
$$
For any finite set of primes $S$, we prove that there are normalized Hecke eigenforms which are non-ordinary for each $p\in S$. The proof is elementary and follows from a generalization of work of Choie, Kohnen and the third author
\cite{ChoieKohnenOno}.
\end{abstract}
\maketitle

\section{Introduction and Statement of Results}

If $k\geq 4$ is even, then let $M_k$ (resp. $S_k$) denote the finite dimensional $\C$-vector space of weight $k$
holomorphic modular forms (resp. cusp forms) on $\SL_2(\Z)$. Furthermore, let $M_k^{!}$ denote the infinite dimensional
space of weakly holomorphic modular forms of weight $k$ with respect to $\SL_2(\Z)$. Recall that a meromorphic modular form is weakly holomorphic if its poles (if any) are supported at cusps. We shall identify a modular form on
$\SL_2(\Z)$ by its Fourier expansion at infinity
$$
f(z)=\sum_{n\gg -\infty} a_f(n)q^n,
$$
where $q:=e^{2\pi i z}$.

Suppose that $O_L$ is the ring of integers of a number field $L$, and suppose that
$$
f(z)=\sum_{n=1}^{\infty}a_f(n)q^n\in S_k\cap O_L[[q]]
$$
is a normalized Hecke eigenform for $\SL_2(\Z)$.
We say that $f$ is non-ordinary at a prime $p$ if there is a prime ideal $\frak{p}\subset O_L$ above $p$ for which
$$
a_f(p)\equiv 0\pmod{\frak{p}}.
$$
Very little is known about the distribution of non-ordinary primes. We recall the following well-known open problem
(see Gouv\^ea's expository article \cite{Gouvea}).

\begin{problem}
Are there infinitely many non-ordinary primes for a generic normalized Hecke eigenform $f(z)$?
\end{problem}

We do not solve this problem here. It remains open. However, we establish the following related result.

\begin{theorem}\label{MainThm} If $S$ is a finite set of primes, then there are infinitely many normalized Hecke eigenforms
for $\SL_2(\Z)$ which are non-ordinary for each $p\in S$.
\end{theorem}

\begin{remark}
The proof of Theorem~\ref{MainThm} relies on a general theorem about the Fourier coefficients of weakly holomorphic modular forms modulo $p$ (see Theorem~\ref{TechnicalTheorem}). For normalized Hecke eigenforms, this general result
incorporates classical results of Hatada \cite{Hatada} (in the case where $p=2$ and $3$) and Hida \cite{Hida1, Hida2, Hida3} (for primes
$p\geq 5$) on non-ordinary primes.
\end{remark}

\begin{remark}
The proof of Theorem~\ref{MainThm} is constructive. Suppose that $S=\{p_1, p_2,\dots, p_m\}$
is a finite set of primes. Suppose that $k\geq 12$ is an even integer. If for each $p\in S$ there is a choice of $t\in A=\{4, 6, 8, 10, 14\}$ for which $(p-1)|(k-t)$, then every prime in $S$ is non-ordinary for every
normalized Hecke eigenform $f\in S_k$. The earlier work of Choie, Kohnen and the third author \cite{ChoieKohnenOno} is eclipsed by this result thanks to the flexibility in the choice of $t$ above.
\end{remark}

In Section~2 we recall certain facts about modular forms, and we prove Theorem~\ref{TechnicalTheorem}. The proof is elementary. In Section~3 we obtain Theorem~\ref{MainThm} as a simple consequence when $p\geq 5$, combining with the known result on $p=2,3$, and in Section~4 we offer some numerical examples.

\section{Preliminaries}

\subsection{Nuts and Bolts}
As usual, let $\Delta(z)\in S_{12}$ be the cusp form
\begin{equation}\label{DeltaFunction}
\Delta(z):=q\prod_{n=1}^{\infty}(1-q^n)^{24}=q-24q^2+\dots,
\end{equation}
and, for even $k\geq 4$, let $E_k(z)\in M_k$ be the normalized Eisenstein series
\begin{equation}\label{Eisenstein}
E_k(z):=1-\frac{2k}{B_k}\sum_{n=1}^{\infty}\left(\sum_{1\leq d\mid n}d^{k-1}\right)q^n,
\end{equation}
where the rational numbers $B_k$ are the usual Bernoulli numbers given by the generating function
$$
\sum_{k=0}^{\infty}B_k\cdot \frac{t^k}{k!}=\frac{t}{e^t-1}=1-\frac{1}{2}t+\frac{1}{12}t^2-\dots.
$$
For convenience, we let $E_0(z):=1$. Finally, we let $j(z)$ be the usual modular function
\begin{equation}\label{jfunction}
j(z):=\frac{E_4(z)^3}{\Delta(z)}=q^{-1}+744+196884q+\dots.
\end{equation}
Finally, for convenience, if $k\in 2\Z$, then throughout we define $\delta(k)\in \{0, 4, 6, 8, 10, 14\}$
so that
\begin{equation}
\delta(k)\equiv k\pmod{12}.
\end{equation}

In the proof, we need the following propositions.
\begin{proposition}\label{proposition1}
 A normalized Hecke eigenform is non-ordinary at $p$
if there is an $m\geq 1$ such that $a_f(p^m)\equiv 0\pmod{p}$.
\end{proposition}

\begin{proof}
This follows from the fact that $T_p f(z)=a_f(p)f(z)$ for every prime $p$ when $f(z)$ is a normalized Hecke eigenform of weight $k$. Here $T_p$ is the $p$-th Hecke operator.
In particular, on prime power exponents, we have
\begin{equation*}
a_f(p)a_f(p^m)=a_f(p^{m+1})+p^{k-1}a_f(p^{m-1})\equiv a_f(p^{m+1}) \ \ ({\rm{mod}}\ p)
\end{equation*}
for every non-negative integer $n$.
By induction, we find that
\begin{equation*}
a_f(p^m)\equiv a_f(p)^m\ \ ({\rm{mod}}\ p).
\end{equation*}
This proves the proposition.
\end{proof}

The following well-known propositions play a central role in the proof of Theorem~\ref{TechnicalTheorem}.

\begin{proposition}\label{proposition2}
If $p\geq 5$ is prime, then as a $q$-series, $E_{p-1}(z)\equiv 1\pmod p$.
\end{proposition}
\begin{proof}
This can be found on page 38 of \cite{OnoCBMS}.
\end{proof}
\begin{proposition}\label{proposition3} If $f(z)=\sum_{n\gg -\infty}a_f(n)q^n\in M_2^{!}$, then $a_f(0)=0$.
\end{proposition}
\begin{proof}
By a simple generalization of Lemma~2.34 of \cite{OnoCBMS}, it is known that
every weakly-holomorphic modular form $h(z)$ of weight $2$ may be represented as $P(j(z))E_{14}(z){\Delta(z)}^{-1}$, where $P(x)$ is a polynomial of $x$.
Dropping the dependence on $z$ for convenience, we have the following well-known identities
\begin{equation*}
-\frac{1}{2\pi i}\frac{d}{dz}j = \frac{E_{14}}{\Delta},
\end{equation*}
\begin{equation*}
j^w\frac{d}{dz}j = \frac{1}{w+1}\frac{d}{dz}j^{w+1},
\end{equation*}
where $w\in \Z_{\geq 0}$.
Therefore, it follows that $h$ is the derivative of a polynomial in $j$, and so its constant term in the Fourier expansion is zero.
\end{proof}

\begin{remark}
For more standard facts about modular forms the reader may see \cite{OnoCBMS}.
\end{remark}

\subsection{Our main technical result}

In 2005 Choie, Kohnen and the third author proved the following (see Corollary 1.3 of
\cite{ChoieKohnenOno}). This result recovered earlier aforementioned results of Hatada and Hida.

\begin{theorem}\label{OldTheorem}
Let $p$ be a prime, and suppose that $f(z)=\sum_{n=1}^{\infty}a_f(n)q^n\in S_k$ is a normalized Hecke eigenform. Let $L_f$ be the number field generated by the coefficients of $f(z)$, and let $\mathfrak{p}\in O_{L_f}$ be any prime ideal above $p$.

\begin{enumerate}
\item If $p=2, 3$, then
\begin{equation*}
a_f(p)\equiv 0\ \ ({\rm{mod}}\ \mathfrak{p}).
\end{equation*}

\item If $p\geq 5$, $\delta(k)\in\{4, 6, 8, 10, 14\}$ and $k\equiv \delta(k) \pmod{p-1}$, then
\begin{equation*}
a_f(p)\equiv 0\ \ ({\rm{mod}}\ \mathfrak{p}).
\end{equation*}
\end{enumerate}
\end{theorem}

Here we strengthen this result for primes $p\geq 5$ by extending it to all the $k$ without any condition on $\delta(k)$.

\begin{theorem}\label{TechnicalTheorem} Let $p\geq 5$ be prime, and suppose that
 $f(z)=\sum_{n\gg-\infty}^{\infty}a_f(n)q^n\in M_k^{!}\cap O_L[[q]]$, where $k\in 2\mathbb{Z}$ and ${O}_L$ is the ring of algebraic integers of a number field $L$.

\begin{enumerate}
\item Suppose that $a\geq 0$ and $m\in A=\{4, 6, 8, 10, 14\}$ are integers for which  $$k-2\leq (m-2) p^a.$$ If  ${\rm{ord}}_{\infty}(f)>-p^a$ and $(p-1)|(k-m)$, then for any integer $b\geq a$, we have
\begin{equation*}
a_f(p^b)\equiv -\frac{2m}{B_m}a_f(0)\ \    ({\rm{mod}} \  p).
\end{equation*}

\item  Suppose that $k\leq 2$, $r, s\in\mathbb{Z}_{\geq 0}$ and $t, u\in\mathbb{Z}_{>0}$ are integers for which $$2-k=r(p-1)+sp^t,$$
where  $s\neq 2$. If $\ord_{\infty}(f)>-p^u, u\leq t$, then for any integer $v$ such that $u\leq v\leq t$, we have
\begin{equation*}
a_f(p^v)\equiv a_f(0)\equiv 0\ \    ({\rm{mod}} \  p).
\end{equation*}
\end{enumerate}
\end{theorem}

\begin{proof}
The proofs in both cases begin with the construction of suitable weakly-holomorphic modular forms of weight $2-k$. The product of such forms with $f$ have weight 2,
and so Proposition~\ref{proposition3} implies that their constant terms vanish.

For the case $(1)$, first note that $(k-2)-(m-2)p^b\equiv k-m\ ({\rm{mod}} \  {p-1})$.
As we have $(p-1)|(k-m)$ and $k-2\leq (m-2)p^b$,
we may find a non-negative integer $c$ such that
\begin{equation*}
2-k=c(p-1)-(m-2)p^b.
\end{equation*}
Let $g_m$ be the function
\begin{equation*}
g_m:=j\frac{E_6^{(1+i^m)/2}}{E_4^{(m+1+3i^m)/4}}=
\left\{
         \begin{array}{lllll}
         j\frac{E_6}{E_4^2} & {\rm{for}}\ \  m=4 \\
         j\frac{1}{E_4} & {\rm{for}}\ \  m=6 \\
         j\frac{E_6}{E_4^3} & {\rm{for}}\ \  m=8 \\
         j\frac{1}{E_4^2} & {\rm{for}}\ \  m=10 \\
         j\frac{1}{E_4^3} & {\rm{for}}\ \  m=14
         \end{array}
       \right.
\in M_{2-m}^!.       
\end{equation*}
Then we have
\begin{equation*}
g_m^{p^b}E_{p-1}^c\in M_{2-k}^!.
\end{equation*}
That is to say, the constant term of $g_m^{p^b}E_{p-1}^{c}f$ is zero.
From Proposition ~\ref{proposition2} we know that
\begin{equation*}
E_{p-1}\equiv 1\ ({\rm{mod}}\ p).
\end{equation*}
Then we have that constant term of $g_m^{p^b}f$ is zero modulo $p$.
By using Fermat's little theorem to compute the multinomials, we get
\begin{eqnarray*}
\begin{array}{ll}
g_m^{p^b}f
&=(q^{-1}+744+O(q))^{p^b}(1-504q+O(q^2))^{\frac{p^b(1+i^m)}{2}}(1+(-240)q+O(q^2))^{\frac{p^b(m+1+3i^m)}{4}}f  \\
&\equiv (q^{-p^b}+744+O(q^{p^b}))(1-252(1+i^m)q^{p^b}+O(q^{2p^b}))\\ & \ \ \ (1-60(m+1+3i^m)q^{p^b}+O(q^{2p^b}))\sum_{n\gg-\infty}^{\infty}a_f(n)q^n\\
&\equiv (q^{-p^b}+432-60m-432i^m+O(q^{p^b}))\sum_{n\gg-\infty}^{\infty}a_f(n)q^n\ \ ({\rm{mod}}\ p).
\end{array}
\end{eqnarray*}
We already know $\ord_{\infty}(f)>-p^a\geq -p^b$, then we know the constant term $c_{m, p}$ of $g_m^{p^b}f$ must satisfy the following congruence
\begin{equation*}
c_{m, p} \equiv a_f(p^b)+(432-60m-432i^m)a_f(0)\ \ ({\rm{mod}}\ p).
\end{equation*}
As $c_{m, p}$ is known to be zero modulo $p$ and for $m\in A$,
\begin{equation*}
\frac{2m}{B_m}=432-60m-432i^m,
\end{equation*}
we get the conclusion.

For the case $(2)$, as we have $2-k=r(p-1)+sp^t$ and $sp^{t-u}\neq 2$, we can find $c_1, c_2\in\mathbb{Z}_{\geq 0}$ such that $4c_1+6c_2=sp^{t-u}$. Then we have
\begin{equation*}
(E_4^{c_1}E_6^{c_2})^{p^u}E_{p-1}^r f\in M_2^!.
\end{equation*}
Hence we have that the constant term of $(E_4^{c_1}E_6^{c_2})^{p^u}E_{p-1}^r f$ is zero.
As
\begin{equation*}
(E_4^{c_1}E_6^{c_2})^{p^u}E_{p-1}^r f\equiv (1+O(q^{p^u}))f\ \ ({\rm{mod}}\ p)
\end{equation*}
and $\ord_{\infty}(f)>-p^u$, we know $a_f(0)\equiv 0\ ({\rm{mod}}\ p)$.
To prove the case of $a_f(p^v)$, for $u\leq v\leq t$, we may find $c_1', c_2'\in\mathbb{Z}_{\geq 0}$ such that $4c_1'+6c_2'=sp^{t-v}$. Then we have
\begin{equation*}
j^{p^v}(E_4^{c_1'}E_6^{c_2'})^{p^v}E_{p-1}^r f\in M_2^!.
\end{equation*}
Hence the constant term of $j^{p^v}(E_4^{c_1'}E_6^{c_2'})^{p^v}E_{p-1}^r f$ is zero. As
\begin{equation*}
(jE_4^{c_1'}E_6^{c_2'})^{p^v}E_{p-1}^r f\equiv (q^{-p^v}+744+240c_1'-504c_2'+O(q^{p^v}))f\ \ ({\rm{mod}}\ p)
\end{equation*}
and $\ord_{\infty}(f)>-p^u\geq -p^v$, we get
\begin{equation*}
a_f(p^v)+(744+240c_1'-504c_2')a_f(0)\equiv 0\ \ ({\rm{mod}}\ p).
\end{equation*}
Knowing that $a_f(0)\equiv 0\ ({\rm{mod}}\ p)$, we get the conclusion.
\end{proof}

\section{Proof of Theorem~\ref{MainThm}} By Theorem~\ref{OldTheorem}, $p=2$ and $3$ are non-ordinary for every normalized Hecke eigenform on $\SL_2(\Z)$.
Therefore, we may assume that $S$ consists only of primes $p\geq 5$.

For the given finite set of primes $S$, let $k_S(j,m):=j\prod_{p\in S} (p-1)+m$, where $j$ is an arbitrary non-negative integer, $m\in A$. For each $j$ and $m$ let  $b_S(j,m)$ be any integer for which $$k_S(j,m)-2<(m-2)p^{b_S(j,m)}$$ for all $p\in S$. Let $f=\sum_{n=1}^{\infty}a_f(n)q^n$ be any Hecke eigenform of weight $k_S(j,m)$. By Theorem~\ref{TechnicalTheorem} (1), since $a_f(0)=0$, we have
\begin{equation*}
a_f(p^{b_S(j,m)})\equiv 0\ \ ({\rm{mod}}\ p)
\end{equation*}
for all $p\in S$.
Applying Proposition ~\ref{proposition1}, we know that $f$ is non-ordinary for each $p\in S$. As $j$ can be chosen freely, we get the conclusion.

\section{Examples}

\begin{example}
Let $S=\{2, 3, 5, 7, 11, 13, 17, 19\}$. In the following table we list some of the weights $k$ for which Hecke eigenforms are non-ordinary at each prime $p$.
\begin{center}
    \begin{tabular} {| l | l l l l l l l l l l l l l l l l|}
    \hline
     $p$ &  \multicolumn{16}{|c|}{$12\leq k\leq 42$ such that all Hecke eigenforms $S_k$ are non-ordinary at $p$}\\ \hline
    \centering{2} & 12 & 14 & 16 & 18 & 20 & 22 & 24 & 26 & 28 & 30 & 32 & 34 & 36 & 38 & 40 & 42 \\ \hline
    3 & 12 & 14 & 16 & 18 & 20 & 22 & 24 & 26 & 28 & 30 & 32 & 34 & 36 & 38 & 40 & 42 \\ \hline
    5 & 12 & 14 & 16 & 18 & 20 & 22 & 24 & 26 & 28 & 30 & 32 & 34 & 36 & 38 & 40 & 42 \\ \hline
    7 & 12 & 14 & 16 & 18 & 20 & 22 & 24 & 26 & 28 & 30 & 32 & 34 & 36 & 38 & 40 & 42 \\ \hline
    11 &   & 14 & 16 & 18 & 20 &    & 24 & 26 & 28 & 30 &    & 34 & 36 & 38 & 40 &    \\ \hline
    13 &   & 14 & 16 & 18 & 20 & 22 &    & 26 & 28 & 30 & 32 & 34 & & 38 & 40 & 42 \\ \hline
    17 &   & 14 &    &    & 20 & 22 & 24 & 26 &    & 30 &    &    & 36 & 38 & 40 & 42 \\ \hline
    19 &   & 14 &    &    &    & 22 & 24 & 26 & 28 &    & 32 &    & &  & 40 & 42\\ \hline
    \end{tabular}
\end{center}
In particular, we consider the case $k=26$ and check its non-ordinariness.
We have the following $q$-expansion of the normalized weight $26$ Hecke eigenform $f_{26}=\Delta E_6 E_4^2$,
\begin{eqnarray*}
\begin{array}{ll}
f_{26}(z)=
& q-48q^2-195804q^3-33552128q^4-741989850q^5+9398592 q^6+39080597192 q^7\\
& +3221114880 q^8-808949403027q^9+35615512800 q^{10}+8419515299052 q^{11}\\
& +6569640870912 q^{12}-81651045335314q^{13}-1875868665216q^{14}\\
&+145284580589400 q^{15}+1125667983917056 q^{16}-2519900028948078q^{17}\\
& +38829571345296 q^{18}-6082056370308940q^{19}+O(q^{20}).
\end{array}
\end{eqnarray*}
We can easily check that $a_{f_{26}}(p)\equiv 0$ (mod $p$) for each $p\in S$. Of course we can also choose weights $k$ of the form $k=26+720j$, for every $j\in \N$. Note that $720=[5-1, 7-1, 11-1, 13-1, 17-1, 19-1]$.
\end{example}

\end{document}